\documentclass[12pt]{article}
\usepackage{graphicx}
\usepackage{color,latexsym,amsfonts,amssymb}
\usepackage{bm}     %use \bm instead of \mathbf
\usepackage{amsthm} %command for \proof
\usepackage{amsmath}
\usepackage{booktabs} %table for 3 lines, toprule/midrule/bottomrule
\usepackage{boxedminipage}

\DeclareMathOperator*{\argmin}{argmin}
\DeclareMathOperator*{\diag}{diag}

\newtheorem{theorem}{Theorem}

\newtheorem{lemma}{Lemma}

\title{A Generalized Fundamental Matrix for Computing Fundamental Quantities of Markov Systems}
\author{Li Xia\thanks{L. Xia is with the Center for Intelligent and Networked Systems (CFINS), Department of Automation, TNList, Tsinghua University, Beijing 100084, China (Email:
xial@tsinghua.edu.cn)}, \ Peter W. Glynn\thanks{P. W. Glynn is with
the Department of Management Science and Engineering, Stanford
University, Stanford, CA 94305, USA (Email: glynn@stanford.edu)}}
\begin{document}
\maketitle
\begin{abstract}
As is well known, the fundamental matrix $(\bm I - \bm P + \bm e \bm
\pi)^{-1}$ plays an important role in the performance analysis of
Markov systems, where $\bm P$ is the transition probability matrix,
$\bm e$ is the column vector of ones, and $\bm \pi$ is the row
vector of the steady state distribution. It is used to compute the
performance potential (relative value function) of Markov decision
processes under the average criterion, such as $\bm g=(\bm I - \bm P
+ \bm e \bm \pi)^{-1} \bm f$ where $\bm g$ is the column vector of
performance potentials and $\bm f$ is the column vector of reward
functions. However, we need to pre-compute $\bm \pi$ before we can
compute $(\bm I - \bm P + \bm e \bm \pi)^{-1}$. In this paper, we
derive a generalization version of the fundamental matrix as $(\bm I
- \bm P + \bm e \bm r)^{-1}$, where $\bm r$ can be any given row
vector satisfying $\bm r \bm e \neq 0$. With this generalized
fundamental matrix, we can compute $\bm g=(\bm I - \bm P + \bm e \bm
r)^{-1} \bm f$. The steady state distribution is computed as $\bm
\pi = \bm r(\bm I - \bm P + \bm e \bm r)^{-1}$. The Q-factors at
every state-action pair can also be computed in a similar way. These
formulas may give some insights on further understanding how to
efficiently compute or estimate the values of $\bm g$, $\bm \pi$,
and Q-factors in Markov systems, which are fundamental quantities
for the performance optimization of Markov systems.
\end{abstract}

\section{Introduction}\label{section_intro}
Markov decision processes (MDPs) are widely adopted to model the
dynamic decision problem in stochastic systems
\cite{Bertsekas12,Puterman94}. The fundamental matrix $(\bm I - \bm
P + \bm e \bm \pi)^{-1}$ plays a key role in the performance
optimization of MDPs. With the fundamental matrix, we can further
study the properties of Markov systems. For example, we can use it
to compute the performance potential (or called relative value
function) of MDPs under the average criterion, i.e., $\bm g=(\bm I -
\bm P + \bm e \bm \pi)^{-1} \bm f$.

The concept of the fundamental matrix was first proposed by J. G.
Kemeny in his book ``\emph{Finite Markov Chains}" coauthored with L.
J. Snell in 1960 \cite{Kemeny60}. In this book, the fundamental
matrix is defined as $\bm Z^*:=(\bm I - \bm P + \bm e \bm
\pi)^{-1}$. Many analysis, such as the mean passage time and the
variance of passage time, can be conducted by using this fundamental
matrix. The fundamental matrix also plays a key role in the
performance sensitivity analysis of Markov systems. The original
work about the sensitivity analysis of the steady state distribution
and the fundamental matrix with respect to the stochastic matrix of
Markov systems can be referred backward to P. Schweitzer's work in
1968 \cite{Schweitzer68}. P. Schweitzer presented a perturbation
formalism that shows how the stationary distribution and the
fundamental matrix of a Markov chain containing a single irreducible
set of states change as the transition probabilities vary. The
sensitivity information can be represented in a series of Rayleigh
perturbation expansions of the fundamental matrix and other
parameters. This is the main target of the perturbation analysis of
Markov chains at the early stage.

E. Seneta and C. D. Meyer did a lot of work \cite{Meyer94,Seneta93}
to study the relation between the eigenvalues of stochastic matrix
$\bm P$ and the condition number ($\max\{|a^\#_{i,j}|\}$) of group
generalized inverse ($\bm A^\#=(\bm I - \bm P + \bm e \bm \pi)^{-1}
- \bm e \bm \pi$) of matrix $\bm A = \bm I - \bm P$, where
$a^\#_{i,j}$ is the element of matrix $\bm A^\#$. Some inequalities
are derived to quantify the sensitivity of the steady state
distribution when $\bm P$ is perturbed to $\bm P'$. Therefore, the
sensitivity of the steady state distribution can be analyzed through
studying the eigenvalues of stochastic matrix $\bm P$. This is the
main idea of the perturbation analysis of Markov chains at that
period. More than the sensitivity analysis of the steady state
distribution, X. R. Cao proposed the sensitivity-based optimization
theory that focuses on the sensitivity analysis of the system
performance with respect to the perturbed transition probabilities
or policies \cite{Cao97,Cao07}. This approach works well for
different system settings, including the average or discounted
criterion, the unichain or multichain, Markov or semi-Markov
systems.

Performance potential $\bm g$ is a fundamental quantity in MDPs. We
have to compute or estimate its value before we conduct the policy
iteration or sensitivity-based optimization. For an MDP under the
discounted criterion, we can directly compute it as $\bm g= (\bm I -
\gamma \bm P)^{-1}\bm f$ since the matrix $(\bm I - \gamma \bm P)$
is invertible \cite{Puterman94}, where $\gamma$ is the discount
factor and $0<\gamma<1$. For an MDP under the average criterion, the
fundamental matrix is used to compute the value of performance
potentials and it has the form $\bm g=(\bm I - \bm P + \bm e \bm
\pi)^{-1} \bm f$. Since the fundamental matrix can be decomposed as
$(\bm I - \bm P + \bm e \bm \pi)^{-1} = \sum_{n=0}^{\infty}(\bm P^n
- \bm e \bm \pi)$, we can rewrite the definition of the performance
potential as $\bm g= \sum_{n=0}^{\infty}(\bm P^n \bm f - \eta \bm e
)$, where $\eta = \bm \pi \bm f$ is the long-run average
performance. That is, we can derive the following sample path
version of the performance potential as $g(i) = E\{
\sum_{n=0}^{N}(f(X_n) - \hat{\eta})|X_0=i \}$, where $X_n$ is the
system state at time $n$, $\hat{\eta}$ is the estimated value of the
long-run average performance, $i \in \mathcal S$ is a state of the
state space $\mathcal S$, and $N$ is a proper integer that is to
control the estimation variance \cite{Cao07}. However, we see that
when we compute $\bm g=(\bm I - \bm P + \bm e \bm \pi)^{-1} \bm f$,
we have to pre-compute the value of $\bm \pi$ first. This issue was
also pointed out by J. G. Kemeny when he studied the computation of
the fundamental matrix $\bm Z^*=(\bm I - \bm P + \bm e \bm
\pi)^{-1}$. He wrote ``It also suffers from the difficulty that one
must compute $\alpha$ (solving $n$ equations) before one can compute
$\bm Z^*$" \cite{Kemeny81}, where $\alpha$ is the steady state
distribution $\bm \pi$ in our paper.

In this paper, we study another form of the fundamental matrix as
$\bm Z_r := (\bm I - \bm P + \bm e \bm r)^{-1}$, where $\bm r$ is
any row vector satisfying $\bm r \bm e \neq 0$. Using this
generalized fundamental matrix, we can compute the value of
performance potentials as $\bm g = (\bm I - \bm P + \bm e \bm
r)^{-1} \bm f$, where all the parameters are known and no
pre-computation is required. The traditional approach $\bm g = (\bm
I - \bm P + \bm e \bm \pi)^{-1} \bm f$ is a special case where we
choose $\bm r = \bm \pi$. We can also choose $\bm r$ as other
vectors. This formula can also be used to compute the value of $\bm
\pi$ as $\bm \pi = \bm r (\bm I - \bm P + \bm e \bm r)^{-1}$, where
the kernel computation remains the same as the generalized
fundamental matrix $(\bm I - \bm P + \bm e \bm r)^{-1}$.

There exist two works about the generalization of the fundamental
matrix in the literature. After proposing the concept of the
fundamental matrix, J. G. Kemeny further studied a generalized form
\cite{Kemeny81}. A special case for ergodic Markov chains is that he
defined $\bm Z_{\beta} := (\bm I - \bm P + \bm e \bm \beta)^{-1}$,
$\bm \beta \bm e = 1$, where $\bm \beta$ is a row vector. In
Kemeny's work, it is shown that the above matrix can be used to
compute the stationary distribution as $\bm \pi = \bm \beta \bm
Z_{\beta}$. Similar result was later reported in J. J. Hunter's book
\cite{Hunter83}, which has the form $\bm \pi = \bm u (\bm I - \bm P
+ \bm e \bm u)^{-1}$, $\bm u \bm e \neq 0$, where $\bm u$ is a row
vector. After that, J. J. Hunter further gave a thorough study on
varied forms of general inverse of Markovian kernel $(I-P)$ in his
recent works \cite{Hunter07,Hunter14}. One form of the general
inverse is written as $(\bm I - \bm P + \bm t \bm u)^{-1}$ with
condition $\bm \pi \bm t \neq 0$ and $\bm u \bm e \neq 0$, where
$\bm t$ is a column vector. Obviously, $(\bm I - \bm P + \bm t \bm
u)^{-1}$ is more general than the above results. Although these
works have a similar result to ours, they focus on the computation
of the steady state distribution or the mean first passage time.
There is no study on the computation of performance potentials or
relative value functions. Compared with the widely-adopted form $\bm
g = (\bm I - \bm P + \bm e \bm \pi)^{-1} \bm f$, our new formula
$\bm g = (\bm I - \bm P + \bm e \bm r)^{-1} \bm f$ does not require
the pre-computation of $\bm \pi$. It may shed some light on how to
efficiently compute or estimate the value of $\bm g$, which is an
essential procedure for the policy iteration in MDPs. Moreover, we
also study the generalization of the fundamental matrix not only for
a discrete time Markov chain, but also for a continuous time Markov
process. We further extend the similar idea to the representation
and computation of Q-factors that are fundamental quantities for
reinforcement learning and artificial intelligence \cite{Silver16,
Sutton98}.

\section{Main Results}\label{section_result}

\subsection{Generalized Fundamental Matrix}\label{subsection_GFM}
We focus on the discussion of a Markov chain with finite states. The
state space is denoted as $\mathcal S:=\{1,2,\cdots,S\}$. The
transition probability matrix is denoted as $\bm P$ and its element
is $P(i,j)$ indicating the probability of which the system transits
from the current state $i$ to the next state $j$, where $i,j\in
\mathcal S$. The steady state distribution of this Markov chain is
denoted as an $S$-dimensional row vector $\bm \pi$ and its element
$\pi(i)$ is the probability of the system staying at state $i$, $i
\in \mathcal S$. Obviously, we have
\begin{equation}
\bm P \bm e = \bm e, \qquad \bm \pi \bm P = \bm \pi, \qquad \bm \pi
\bm e = 1,
\end{equation}
where $\bm e$ is an $S$-dimensional column vector with all elements
1.

First, we give the following lemma about shifting the eigenvalues of
a general matrix.
\begin{lemma}\label{lemma1}
Suppose that $\bm A$ is a square matrix for which $\lambda$ is an
eigenvalue having multiplicity 1, and the associated column
eigenvector is $\bm v$. Let $\lambda_i$ be the other eigenvalues of
$\bm A$ and $\bm \phi_i$ be the associated row eigenvectors. Then,
for any row vector $\bm r$, $\lambda + \bm r\bm v$ is an eigenvalue
of $\bm A + \bm v\bm r$ and the associated column eigenvector is
$\bm v$; other eigenvalues of $\bm A + \bm v\bm r$ are given by
$\lambda_i$ with the associated row eigenvectors $\bm \phi_i$.
\end{lemma}
\begin{proof}
Since $\bm \phi_i$ is a \emph{row eigenvector} of $\bm A$ associated
with eigenvalue $\lambda_i$, we have
\begin{equation}
\bm \phi_i \bm A \bm v = \lambda_i \bm \phi_i \bm v.
\end{equation}
On the other hand, since $\bm v$ is a \emph{column eigenvector} of
$\bm A$, we have
\begin{equation}
\bm \phi_i \bm A \bm v = \lambda \bm \phi_i \bm v.
\end{equation}
Since $\lambda$ has multiplicity 1, $\lambda_i$ is not equal to
$\lambda$. Comparing the above two equations, we directly have
\begin{equation}\label{eq_phiv}
\bm \phi_i \bm v = 0.
\end{equation}
Below, we further study the eigenvalue of matrix $\bm A + \bm v\bm
r$. Using (\ref{eq_phiv}), we have
\begin{equation}
\bm \phi_i (\bm A + \bm v \bm r) = \bm \phi_i \bm A = \lambda_i \bm
\phi_i.
\end{equation}
Therefore, $\lambda_ i$ continues to be an eigenvalue of $\bm A +
\bm v \bm r$ and $\bm \phi_i$ continues to be the associated row
eigenvector of $\bm A + \bm v \bm r$.

Moreover, we have
\begin{equation}
(\bm A + \bm v\bm r) \bm v = \bm A \bm v + \bm v (\bm r\bm v) =
(\lambda + \bm r\bm v) \bm v,
\end{equation}
Therefore, $\lambda + \bm{rv}$ is a new eigenvalue of $\bm A +
\bm{vr}$ and $\bm v$ continues to be the associated column
eigenvector of $\bm A + \bm{vr}$. The lemma is proved.
\end{proof}

For the eigenvalues of the transition probability matrix of a Markov
chain, we have the following lemma.
\begin{lemma}\label{lemma2}
If $\bm P$ is the stochastic matrix of an irreducible and aperiodic
Markov chain, its spectral radius $\rho(\bm P) = \lambda_{1} = 1$
and $|\lambda_{i}| < 1$ for $i \neq 1$, where $\lambda_{i}$ is the
eigenvalues of $\bm P$ descendingly sorted by their modulus.
Moreover, $\lambda_{1} = 1$ is a simple eigenvalue and the
associated column eigenvector is $\bm x_1 = \bm e$.
\end{lemma}

This lemma can be obtained directly from the \emph{Perron-Frobeniu
theorem} that was separately proposed by Oskar Perron in 1907
\cite{Perron07} and Georg Frobenius in 1908 \cite{Frobenius12}. The
original Perron-Frobeniu theorem aims to study the eigenvalues and
eigenvectors of nonnegative matrix. Since the stochastic matrix is a
special case of nonnegative matrix, we can obtain more specific
properties of stochastic matrix, such as the statement in the above
lemma. The proof of this lemma is ignored and interested readers can
find it from reference books \cite{Berman94,Cao07}.

With the above lemmas, we derive the following theorem about the
eigenvalues of matrix $\bm I - \bm P + \bm e \bm r$.
\begin{theorem}\label{theorem1}
Assume that $\bm P$ is the stochastic matrix of an irreducible and
aperiodic Markov chain. Denote $\lambda_i$ as the eigenvalues of
$\bm P$ in descending order of their modulus, $\bm \phi_i$ and $\bm
x_i$ are the associated row and column eigenvectors, respectively.
Then, for any row vector $\bm r$, the matrix $\bm I - \bm P + \bm e
\bm r$ has the following property: one eigenvalue is $\bm r \bm e$
and the associated column eigenvector is $\bm e$; other eigenvalues
are $1-\lambda_i$ for $i \neq 1$, the associated row eigenvectors
are $\bm \phi_i$, and the associated column eigenvectors are $\bm
x_i + \frac{\bm r \bm x_i}{\lambda_i - \bm r \bm e}\bm e$ if $\bm r
\bm e \neq \lambda_i$.
\end{theorem}
\begin{proof}
With Lemma~\ref{lemma2}, we see that $\lambda_1=1$ and
$|\lambda_{i}| < 1$ for $i \neq 1$. We denote $\bm A := \bm I - \bm
P$. It is easy to verify that the eigenvalues of $\bm A$ are
$1-\lambda_i$ and the associated row eigenvectors are the same as
$\bm \phi_i$. That is, $0$ is an eigenvalue of $\bm A$ with
simplicity 1 and the associated column eigenvector is $\bm e$.
Therefore, by applying Lemma~\ref{lemma1}, we see that $\bm r \bm e$
is an eigenvalue of matrix $\bm I - \bm P + \bm e \bm r$ and the
associated column vector is $\bm e$; other eigenvalues are
$1-\lambda_i$ with the associated row eigenvector $\bm \phi_i$ for
$i \neq 1$. The column eigenvectors of $\bm I - \bm P + \bm e \bm r$
can be verified as follows.
\begin{eqnarray}
&&(\bm I-\bm P+\bm e \bm r) \left(\bm x_i +
\frac{\bm r \bm x_i}{\lambda_i-\bm r\bm e} \bm e \right) \nonumber\\
&=& (\bm I - \bm P)\bm x_i + (\bm I - \bm P)\bm e \frac{\bm r \bm x_i}{\lambda_i-\bm r\bm e} + \bm e \bm r \bm x_i + \bm e \bm r \bm e \frac{\bm r \bm x_i}{\lambda_i-\bm r \bm e} \nonumber\\
&=& \lambda_i \bm x_i + \bm 0 + \bm r \bm x_i \bm e +
\frac{\bm r\bm e}{\lambda_i-\bm r\bm e} \bm r \bm x_i \bm e \nonumber\\
&=& \lambda_i \left( \bm x_i + \frac{\bm r \bm x_i}{\lambda_i-\bm r
\bm e} \bm e \right).
\end{eqnarray}
Therefore, the theorem is proved.
\end{proof}

The eigenvalue of a matrix may be a complex number. With
Lemma~\ref{lemma2}, we can see that the eigenvalues of $\bm P$ have
$|\lambda_i|<1$ for $i \neq 1$ and they are located in the unit
circle in the complex plane, as illustrated by the left sub-figure
of Fig~\ref{fig_circleP}. With Theorem~\ref{theorem1}, we can see
that the eigenvalues of $\bm I - \bm P + \bm e \bm r$ are
$1-\lambda_i$ for $i \neq 1$ and they are also located in the unit
circle illustrated by the right sub-figure of
Fig.~\ref{fig_circleP}.

\begin{figure}[htbp]
\centering
\includegraphics[width=0.9\columnwidth]{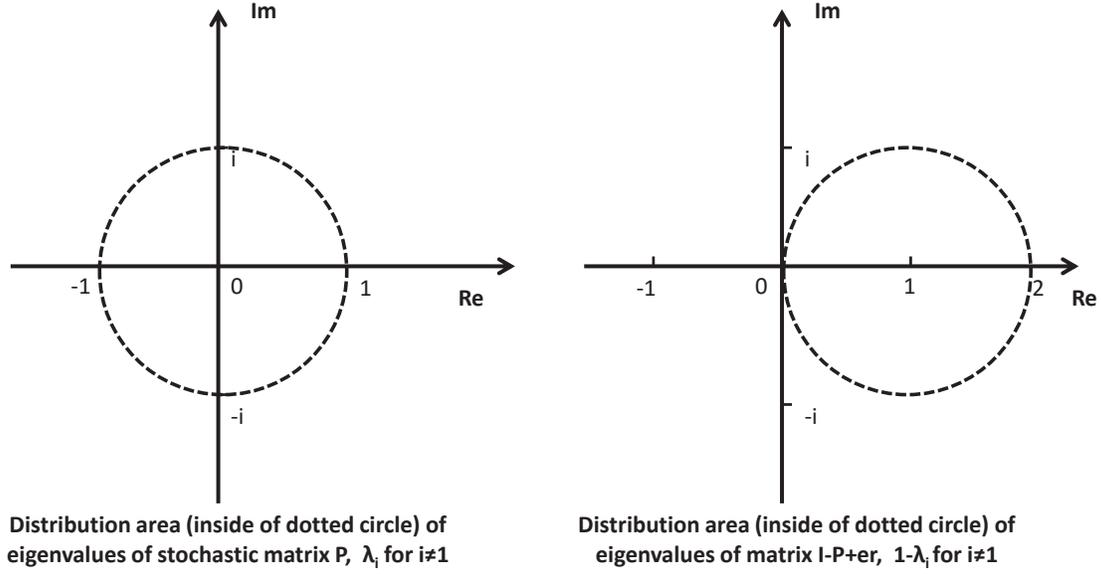}
\caption{The distribution area of eigenvalues of matrix $\bm P$ and
$\bm I - \bm P + \bm e \bm r$, i.e., $\lambda_i$ and $1-\lambda_i$
for $i \neq 1$.}\label{fig_circleP}
\end{figure}

Therefore, we can directly derive the following theorem.
\begin{theorem}\label{theorem2}
The matrix $\bm I-\bm P+\bm e \bm r$ is not singular if and only if
$\bm r \bm e \neq 0$, where $\bm P$ is the stochastic matrix of an
irreducible and aperiodic Markov chain.
\end{theorem}
\begin{proof}
The proof of this theorem is very straightforward. Based on
Theorem~\ref{theorem1} and Lemma~\ref{lemma2}, we see that the
eigenvalues of $\bm I - \bm P + \bm e \bm r$ are either $\bm r \bm
e$ or $1-\lambda_i$ for $i \neq 1$. Since $\bm r \bm e \neq 0$ and
$|\lambda_i|<1$ for $i \neq 1$, we can easily verify that $0$ is not
the eigenvalue of $\bm I - \bm P + \bm e \bm r$. Therefore, the
matrix is invertible and the theorem is proved.
\end{proof}

With Theorem~\ref{theorem2}, we see that $\bm I - \bm P + \bm e \bm
r$ is invertible if $\bm r \bm e \neq 0$. Therefore, we define a
\emph{generalized fundamental matrix} as below.
\begin{equation}\label{eq_Zr}
\bm{Z_r} := (\bm I - \bm P + \bm e \bm r)^{-1}, \qquad \bm r \bm e
\neq 0.
\end{equation}
Compared with the fundamental matrix $\bm Z^*:=(\bm I - \bm P + \bm
e \bm \pi)^{-1}$ defined in the literature \cite{Kemeny60}, $\bm
Z^*$ can be viewed as a special case of $\bm{Z_r}$ with $\bm r = \bm
\pi$.

The generalized fundamental matrix $\bm{Z_r}$ is an important
quantity of Markov chains and it can be utilized to compute the
performance potential and the steady state distribution, as we will
discuss in the following subsections.

\subsection{Computation of Performance Potential}\label{subsection_PP}
As we know, the value function (or performance potential) is an very
important quantity of Markov decision processes. In a standard
policy iteration procedure, we have to compute the value function
for the current policy, which is called the \emph{policy evaluation}
step \cite{Puterman94}. In the approximate dynamic programming, we
study various approximation approaches to simplify the computation
of the value function to alleviate the curse of dimensionality
\cite{Bertsekas12}. Therefore, the efficient computation of the
value function is an very important topic in the field of MDPs. Note
that the computation of value functions under the discount criterion
is easy, because the associated Poisson equation has a unique
solution. We focus on the value function under the long-run average
criterion of MDPs.

The \emph{performance potential} is an alias of the value function
and it has a special physical meaning from the perspective of the
perturbation analysis and the sensitivity-based optimization
\cite{Cao07}. In the following content, we will use the term of
performance potential to study how to compute or estimate it. We
denote the performance potential as an $S$-dimensional column vector
$\bm g$ and its element $g(i)$, $i \in \mathcal S$, is defined as
below.
\begin{equation}\label{eq_g}
g(i) := \lim\limits_{T \rightarrow \infty} E\left\{
\sum_{t=0}^{T-1}[f(X_t) - \eta] | X_0 = i \right\},
\end{equation}
where $X_t$ is the system state at time $t$, $f(X_t)$ is the system
reward at state $X_t$, and $\eta$ is the long-run average
performance defined as below.
\begin{equation}
\eta := \lim\limits_{T \rightarrow \infty} E\left\{ \frac{1}{T}
\sum_{t=0}^{T-1} f(X_t) | X_0 = i \right\} = \lim\limits_{T
\rightarrow \infty} \frac{1}{T} \sum_{t=0}^{T-1} f(X_t),
\end{equation}
where the second equality holds when the Markov chain is a unichain.

Extending the right-hand side of (\ref{eq_g}) at time $t=0$ and
recursively substituting (\ref{eq_g}), we can obtain
\begin{equation}\label{eq_g2}
g(i) = f(i) - \eta + \sum_{j \in \mathcal S} p(i,j)g(j).
\end{equation}
Rewriting the above equation in a matrix form, we obtain the
\emph{Poisson equation} as below.
\begin{equation}\label{eq_poisson}
\bm g = \bm f - \eta \bm e + \bm P \bm g.
\end{equation}
The above equation can be rewritten as below.
\begin{equation}
(\bm I - \bm P)\bm g = \bm f - \eta \bm e.
\end{equation}
However, as we know from Lemma~\ref{lemma2}, the matrix $(\bm I -
\bm P)$ has an eigenvalue with value 0 and it is not invertible.
Noticing the fact that $\bm g + c \bm e$ is still a solution to
(\ref{eq_poisson}) for any constant $c$, we can properly choose $c$
to let $\bm \pi \bm g = \eta$. Therefore, we have
\begin{equation}
(\bm I - \bm P + \bm e \bm \pi)\bm g = \bm f.
\end{equation}
With Theorem~\ref{theorem2}, we see that matrix $(\bm I - \bm P +
\bm e \bm \pi)$ is invertible. The inverse matrix is called the
fundamental matrix defined in the literature \cite{Kemeny60} and it
has
\begin{equation}
(\bm I - \bm P + \bm e \bm \pi)^{-1} = \sum_{n=0}^{\infty} (\bm P -
\bm e \bm \pi)^n = \bm I + \sum_{n=1}^{\infty} (\bm P^n - \bm e \bm
\pi).
\end{equation}
Therefore, the solution to the Poisson equation (\ref{eq_poisson})
can be written as below.
\begin{equation}\label{eq_g3}
\bm g = (\bm I - \bm P + \bm e \bm \pi)^{-1} \bm f.
\end{equation}

The above formula widely exists in the literature \cite{Cao97,Cao07}
and we can use it to numerically compute the value of $\bm g$ for a
specific MDP. Note that the value of $\bm g$ computed with
(\ref{eq_g3}) satisfies the condition $\bm \pi \bm g = \eta$.
However, $\bm \pi$ is not a given parameter in the above equation.
We have to compute the value of $\bm \pi$ before we can use
(\ref{eq_g3}) to compute $\bm g$. This increases the computation
burden. Moreover, if we conduct online estimation, the estimation
error of $\bm \pi$ may increase the estimation variance of $\bm g$.

Fortunately, we have another way to numerically compute $\bm g$
without the extra computation for $\bm \pi$. Since $\bm g + c \bm e$
is still a solution to (\ref{eq_poisson}) for any constant $c$, for
any $S$-dimensional row vector $\bm r$ satisfying $\bm r \bm e \neq
0$, we can choose a proper $c$ to let $\bm r \bm g = \eta$.
Therefore, we can rewrite (\ref{eq_poisson}) as below.
\begin{equation}
(\bm I - \bm P + \bm e \bm r)\bm g = \bm f.
\end{equation}
Since matrix $(\bm I - \bm P + \bm e \bm r)$ is always invertible as
proved in Theorem~\ref{theorem2}, the above equation can be further
rewritten as
\begin{center}
\begin{boxedminipage}{1\columnwidth}
\begin{equation}\label{eq_g4}
\bm g = (\bm I - \bm P + \bm e \bm r)^{-1} \bm f,
\end{equation}
\vspace{-10pt}
\end{boxedminipage}
\end{center}
where $\bm r$ is any $S$-dimensional row vector $\bm r$ satisfying
$\bm r \bm e \neq 0$. We can see that all the parameters in
(\ref{eq_g4}) are given and we can directly compute $\bm g$ with
(\ref{eq_g4}) without any pre-computation.

\noindent \textbf{Remark 1.} Both (\ref{eq_g3}) and (\ref{eq_g4})
are solutions to the Poisson equation (\ref{eq_poisson}) and they
have difference only with a constant column vector $c \bm e$. The
value of $\bm g$ in (\ref{eq_g3}) satisfies $\bm \pi \bm g = \eta$,
while the value of $\bm g$ in (\ref{eq_g4}) satisfies $\bm r \bm g =
\eta$.

\noindent \textbf{Remark 2.} (\ref{eq_g3}) can be viewed as a
special case of (\ref{eq_g4}) if we choose $\bm r = \bm \pi$. With
(\ref{eq_g4}), we have more flexibility to choose different $\bm
r$'s, which may give some insights on the computation or estimation
of $\bm g$.

\subsection{Computation of Steady State Distribution}\label{subsection_SSD}
The fundamental matrix can also be used to compute the steady state
distribution of Markov chains. We also assume that the Markov chain
is irreducible and aperiodic. We know that $\bm \pi$ can be
determined by the following set of linear equations
\begin{equation}
\begin{array}{l}
\bm \pi \bm P = \bm \pi. \\
\bm \pi \bm e = 1.
\end{array}
\end{equation}
We can rewrite the above equations according to the standard form of
linear equations as below.
\begin{equation}
\begin{array}{l}
(\bm I - \bm P^{T}) \bm \pi^{T} = \bm 0. \\
\bm e^T \bm \pi^T = 1.
\end{array}
\end{equation}
That is,
\begin{equation}\label{eq4}
\left[
\left( \begin{array}{ccccc} 1 & 0 & 0 & \cdots & 0 \\
0 & 1 & 0 & \cdots & 0 \\
0 & 0 & 1 & \cdots & 0 \\
\vdots & \vdots & \vdots & \ddots & \vdots \\
0 & 0 & 0 & \cdots & 1 \\
\end{array}
\right) - \left( \begin{array}{ccccc} P_{1,1} & P_{2,1} & P_{3,1} & \cdots & P_{S,1} \\
P_{1,2} & P_{2,2} & P_{3,2} & \cdots & P_{S,2} \\
P_{1,3} & P_{2,3} & P_{3,3} & \cdots & P_{S,3} \\
\vdots & \vdots & \vdots & \ddots & \vdots \\
P_{1,S} & P_{2,S} & P_{3,S} & \cdots & P_{S,S} \\
\end{array}
\right) \right] \left( \begin{array}{c} \pi_1 \\
\pi_2 \\
\pi_3 \\
\vdots \\
\pi_S \\
\end{array} \right) = \left( \begin{array}{c} 0 \\
0 \\
0 \\
\vdots \\
0 \\
\end{array} \right).
\end{equation}
\begin{equation}\label{eq5}
\pi_1 + \pi_2 + \pi_3 + \dots + \pi_S = 1.
\end{equation}
For any $S$-dimensional row vector $\bm r$ satisfying $\bm r \bm e
\neq 0$, we multiply $r(i)$ on both sides of (\ref{eq5}) and summate
this equation to the $i$th equation of (\ref{eq4}),
$i=1,2,\cdots,S$. We can obtain
\begin{equation}
(\bm I - \bm P^{T} + \bm r^T \bm e^T) \bm \pi^{T} = \bm r^T.
\end{equation}

With Theorem~\ref{theorem2}, we know that  $\bm I - \bm P^{T} + \bm
r^T \bm e^T$ is invertible and $(\bm I - \bm P^{T} + \bm r^T \bm
e^T)^{-1} = \bm Z^{T}_{\bm r}$. Therefore, we have
\begin{equation}
\bm \pi^{T} = (\bm I - \bm P^{T} + \bm r^T \bm e^T)^{-1} \bm r^T,
\end{equation}
or
\begin{center}
\begin{boxedminipage}{1\columnwidth}
\begin{equation}\label{eq_pi}
\bm \pi = \bm r (\bm I - \bm P + \bm e \bm r)^{-1}, \qquad \bm r \bm
e \neq 0.
\end{equation}
\vspace{-10pt}
\end{boxedminipage}
\end{center}

From the above equation, we can see that the computation of $\bm
\pi$ has the same key part as the computation of $\bm g$ with
(\ref{eq_g4}), i.e., the computation of the generalized fundamental
matrix $\bm{Z_r}=(\bm I - \bm P + \bm e \bm r)^{-1}$. Therefore,
$\bm{Z_r}$ plays a key role in the analysis of Markov chains.

\subsection{Property Analysis and Estimation Algorithm}\label{subsection_Alg}
In this subsection, we discuss the properties of the generalized
fundamental matrix $\bm{Z_r}$ and the effect on the computation of
performance potentials. If the spectral radius of $\bm P - \bm e \bm
r$ is smaller than 1, we can rewrite the generalized fundamental
matrix as follows.
\begin{equation}\label{eq_Zr2}
\bm{Z_r} = \sum_{n=0}^{\infty}(\bm P - \bm e \bm r)^n, \qquad
\rho(\bm P - \bm e \bm r)<1.
\end{equation}
According to Lemma~\ref{lemma1}, we see that the eigenvalues of $\bm
P - \bm e \bm r$ are $\lambda_1 - \bm r \bm e$ and $\lambda_i$ for
$i \neq 1$, where $\lambda_i$ are the eigenvalues of $\bm P$ sorted
in the descending order of their modulus. With Lemma~\ref{lemma2},
we see that $\lambda_1=1$ and $|\lambda_i|<1$ for $i \neq 1$.
Therefore, we have
\begin{equation}
\rho(\bm P - \bm e \bm r)<1 \quad \Longleftrightarrow \quad 0<\bm r
\bm e <2.
\end{equation}
Furthermore, we can verify that (\ref{eq_Zr2}) can be rewritten as
below.
\begin{equation}
\bm{Z_r} = \sum_{n=0}^{\infty} \left[ \bm P^n - \bm e \bm r
\sum_{j=0}^{n-1} (1 - \bm r \bm e)^j \bm P^{n-1-j} \right], \qquad
0<\bm r \bm e <2.
\end{equation}
If we choose $\bm r$ such that $\bm r \bm e = 1$, then we can
further simplify the above equation as below.
\begin{equation}\label{eq_Zr3}
\bm{Z_r} = \sum_{n=0}^{\infty} \left[ \bm P^n - \bm e \bm r \bm
P^{n-1} \right], \qquad \bm r \bm e = 1,
\end{equation}
where we define $\bm e \bm r \bm P^{n-1} = 0$ when $n=0$. If we
choose $\bm r$ stochastic, then the $(i,j)$'th entry of $(\bm I -
\bm P + \bm e \bm r)^{-1}$ has the interpretation that it is a sum
over $n$ in which the $n$'th term (for $n \geq 1$) is $P(X_n = j |
X_0 = i) - P( X_{n-1} = j | X_0$ having initial distribution $\bm
r$). If we manipulate the terms in the summation of (\ref{eq_Zr3}),
we can obtain
\begin{equation}\label{eq_Zr4}
\bm{Z_r} = \sum_{n=0}^{\infty} \left[ \bm P^n - \bm e \bm r \bm
P^{n} \right] + \lim\limits_{n \rightarrow \infty}\bm e \bm r \bm
P^n, \qquad \bm r \bm e = 1.
\end{equation}
Since the Markov chain is irreducible, aperiodic, and finite, the
limiting probability exists and it equals the steady state
distribution. That is,
\begin{equation}
\lim\limits_{n \rightarrow \infty} \bm P^n = \bm e \bm \pi.
\end{equation}
Therefore, the above equation (\ref{eq_Zr4}) can be rewritten as
below.
\begin{equation}\label{eq_Zr5}
\bm{Z_r} = \sum_{n=0}^{\infty} \left[ \bm P^n - \bm e \bm r \bm
P^{n} \right] + \bm e \bm \pi, \qquad \bm r \bm e = 1.
\end{equation}
Therefore, neglecting the steady distribution $\bm e \bm \pi$ for
simplicity, we can see that the $n$'th term of the above summation
is $P(X_n = j | X_0 = i) - P( X_n = j | X_0$ having initial
distribution $\bm r$). This interpretation can help develop online
estimation algorithms for quantities related to $\bm{Z_r}$ from a
viewpoint of sample paths.

Substituting (\ref{eq_Zr5}) into (\ref{eq_g4}), we have
\begin{eqnarray}
\bm g &=& (\bm I - \bm P + \bm e \bm r)^{-1} \bm f \nonumber\\
&=& \sum_{n=0}^{\infty} \left[ \bm P^n - \bm e \bm r \bm P^{n}
\right] \bm f + \bm e \bm \pi \bm f \nonumber\\
&=& \sum_{n=0}^{\infty} \bm P^n \bm f - \bm e \bm r
\sum_{n=0}^{\infty} \bm P^{n} \bm f + \eta \bm e,
\end{eqnarray}
where $\bm r \bm e = 1$ and $\bm r \bm g = \eta$. Since $\bm g + c
\bm e$ is still a performance potential for any constant $c$, we can
neglect the term $\eta \bm e$ in the above equation and rewrite it
as below.
\begin{equation}\label{eq_42}
\bm g = \sum_{n=0}^{\infty} \bm P^n \bm f - \bm e \bm r
\sum_{n=0}^{\infty} \bm P^{n} \bm f,
\end{equation}
where $\bm r \bm e = 1$ and $\bm r \bm g = 0$.

From the above equation, we can see that $\sum_{n=0}^{\infty} \bm
P^n \bm f$ equals the expectation of the accumulated rewards along
the sample path, i.e., $\mathbb E\{\sum_{t=0}^{\infty} f(X_t)\}$. We
denote
\begin{equation}
\tilde{\bm g} = \sum_{n=0}^{\infty} \bm P^n \bm f = \mathbb E\left\{
\sum_{t=0}^{\infty} f(X_t) \right\}.
\end{equation}
or
\begin{equation}
\tilde{\bm g}_T = \mathbb E\left\{ \sum_{t=0}^{T} f(X_t) \right\},
\end{equation}
for a large constant $T$. We can rewrite (\ref{eq_42}) as below.
\begin{equation}
\bm g = \tilde{\bm g} - \bm e \bm r \tilde{\bm g}.
\end{equation}
When $\bm r$ is a probability distribution, the physical meaning of
the above equation is that: we sum all the rewards along the sample
path, then we use a weighting vector $\bm r$ to obtain a
\emph{reference level} $\bm r \tilde{\bm g}$, the gap between
$\tilde{\bm g}$ and the reference level $\bm r \tilde{\bm g}$ is
exactly the performance potential (this is also the reason that we
call it \emph{relative} value function).

The insight for the estimation algorithm is that we can just sum all
the rewards along the sample path, then we choose an arbitrary
reference level determined by a combination of elements of
$\tilde{\bm g}$, the gap between them is the estimate of performance
potentials. This can help to simplify the online estimation
procedure. $\bm r = \bm \pi$ is one of the special cases. For
example, we can also set $\bm r = (1, 0, 0, \cdots, 0)$, which lets
$\tilde{g}(1)$ as the reference level.

Instead of numerically computing $\bm g$ with (\ref{eq_g4}), we can
further develop an iterative algorithm to estimate the value of $\bm
g$ based on (\ref{eq_g4}). With (\ref{eq_g4}), we have
\begin{equation}
\bm g = \bm f - \bm e \bm r \bm g + \bm P \bm g.
\end{equation}
Suppose $\hat{\bm g}$ is an unbiased estimate of $\bm g$, i.e.,
$E(\hat{\bm g}) = \bm g$. We have
\begin{equation}
E(\hat{\bm g}) = \bm f - \bm e \bm r E(\hat{\bm g}) + \bm P
E(\hat{\bm g}).
\end{equation}
Rewriting the above equation as a \emph{sample path version}, we
derive the following equation that holds from the sense of
statistics
\begin{eqnarray}
E[\hat{g}(X_t)] &=& E[f(X_t) - \bm r \hat{\bm g} + \hat{g}(X_{t+1})]
\nonumber\\
&=& E[f(X_t)] - \sum_{i \in \mathcal S}r(i)E[\hat{g}(i)] +
E[\hat{g}(X_{t+1})].
\end{eqnarray}

Based on the above equation, we develop a \emph{least-squares}
algorithm that can online estimate $\bm g$ based on sample paths.
Define a quantity $z_{t}$ as below, which can be viewed as the new
information learned from the current feedback of the system.
\begin{equation}
z_{t} := f(X_t) - \sum_{i \in \mathcal S}r(i) \hat{g}(i) +
\hat{g}(X_{t+1}) - \hat{g}(X_{t}).
\end{equation}
With a stochastic approximation framework, we have the following
formula to update the value of $\hat{g}$
\begin{equation}\label{eq_lsg}
\hat{g}(X_t) \leftarrow \hat{g}(X_t) + \alpha_t z_{t},
\end{equation}
where $\alpha_t$ is a positive step-size that satisfies the
convergence condition of Robbins-Monro stochastic approximation,
i.e.,
\begin{equation}
\sum_{t=0}^{\infty}\alpha_t = \infty, \quad \sum_{t=0}^{\infty}
\alpha^2_t < \infty.
\end{equation}
or $\alpha_t$ satisfies an even looser condition as below
\cite{Chen14}.
\begin{equation}
\sum_{t=0}^{\infty}\alpha_t = \infty, \quad \lim\limits_{t
\rightarrow \infty} \alpha_t \rightarrow 0.
\end{equation}

The name of least-squares of update formula (\ref{eq_lsg}) comes
from the fact that we aim to obtain the following estimate of $\bm
g$ that can obtain the least squares of $z_t$ in statistics, i.e.,
\begin{equation}
\bm g = \argmin\limits_{\hat{\bm g}} E\{z^2_t\}
\end{equation}

The idea of least-squares update formula (\ref{eq_lsg}) is similar
to that of temporal-difference (TD) algorithm in reinforcement
learning \cite{Sutton98}. Below, we give a procedure framework of
the online estimation algorithm as illustrated in
Fig.~\ref{fig_algo}. In Fig.~\ref{fig_algo}, the algorithm stopping
criterion can be set as the norm of two successive estimates  is
smaller than a given small threshold $\epsilon$, i.e., $\parallel
\bm g - \bm g' \parallel < \epsilon$, or just simply stop the
algorithm after reaching a large step number $T$, i.e., $t > T$.

\begin{figure}
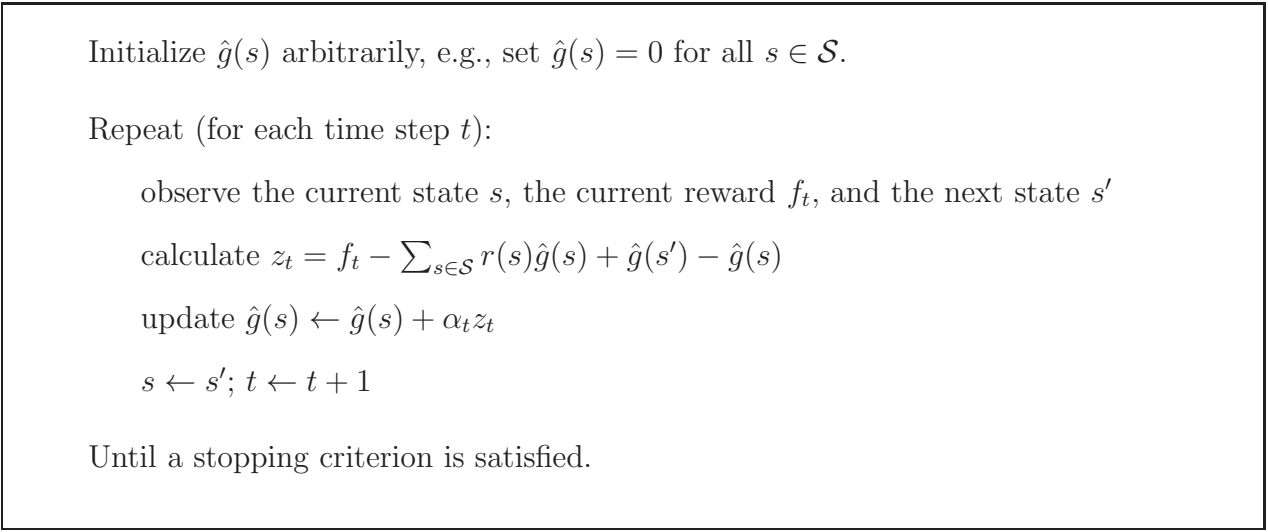

\begin{center}
\begin{boxedminipage}{1\columnwidth}
\vspace{5pt}
\begin{itemize}

\item[] Initialize $\hat{g}(s)$ arbitrarily, e.g., set $\hat{g}(s)=0$
for all $s \in \mathcal S$.

\item[] Repeat (for each time step $t$):

\subitem observe the current state $s$, the current reward $f_t$,
and the next state $s'$

\subitem calculate $z_t = f_t - \sum_{s \in \mathcal
S}r(s)\hat{g}(s) + \hat{g}(s') - \hat{g}(s)$

\subitem update $\hat{g}(s) \leftarrow \hat{g}(s)+\alpha_t z_t$

\subitem $s \leftarrow s'$; $t \leftarrow t+1$

\item[] Until a stopping criterion is satisfied.

\end{itemize}
\vspace{5pt}
\end{boxedminipage}
\end{center}
\caption{An online least-squares algorithm to estimate $\bm g$ based
on (\ref{eq_g4}).}\label{fig_algo}
\end{figure}

\section{Extension to Other Cases}\label{section_Ext}
In the previous section, we discuss the generalized fundamental
matrix for the discrete time Markov chain. In this section, we
extend the result to other cases, including the continuous time
Markov process and the Q-factors in reinforcement learning.

\subsection{Continuous Time Markov Process}\label{subsection_CTMP}
Consider a continuous time Markov process with transition rate
matrix $\bm B$. Assume the Markov process is ergodic, we can
similarly prove
\begin{theorem}\label{theorem5}
The eigenvalue of matrix $\bm B$ has the following property:
$\lambda^B_1 = 0$ is a simple eigenvalue and the associated column
eigenvector is $\bm x^B_1 = \bm e$; $ |\lambda^B_i + \gamma| <
\gamma$ for $i \neq 1$, where $\gamma$ is any constant satisfying
$\gamma \geq \max_{s}|B(s,s)|$.
\end{theorem}
\begin{proof}
By using the uniformization technique in MDPs, we can define a
matrix $\bm P$ as below.
\begin{equation}
\bm P := \bm I + \frac{\bm B}{\gamma}.
\end{equation}
Obviously, $\bm P$ is a stochastic matrix since it satisfies $\bm P
\bm e = \bm e$ and all of its elements are nonnegative. The
statistical behavior of the Markov chain with transition probability
matrix $\bm P$ is equivalent to that of the Markov process with
transition rate matrix $\bm B$ \cite{Puterman94}. We have
\begin{equation}\label{eq28}
\bm B = \gamma(\bm P - \bm I).
\end{equation}
Since the Markov process with $\bm B$ is ergodic, the equivalent
Markov chain with $\bm P$ is also ergodic. Therefore, $\bm P$ is an
ergodic stochastic matrix. With Lemma~\ref{lemma2}, we know that the
eigenvalue of $\bm P$ has
\begin{equation}
\left\{
\begin{array}{ll}
\lambda^P_i = 1, &i=1;\\
|\lambda^P_i| < 1, &i \neq 1;\\
\end{array}
\right.
\end{equation}
Therefore, with (\ref{eq28}), we see that the eigenvalue of $\bm B$
is $\lambda^B_i = \gamma(\lambda^P_i - 1)$ and the column
eigenvector is $\bm x^B_i = \bm x^P_i$, $\forall i$. We have
\begin{equation}
\left\{
\begin{array}{lll}
\lambda^B_i = 0, &\bm x^B_i=\bm e, &i=1;\\
|\lambda^B_i+\gamma| < \gamma, &\bm x^B_i=\bm x^P_i, &i \neq 1;\\
\end{array}
\right.
\end{equation}
The theorem is proved.
\end{proof}

With the above theorem, we know that the eigenvalue of $\bm B$ is
either $0$ or distributed inside of the dotted circle in
Fig.~\ref{fig_circleB}. When $\gamma$ is smaller, we can obtain a
tighter area describing the distribution area of $\lambda^B_i$.
Obviously, the smallest value of $\gamma$ is
\[
\gamma = \max_{s \in \mathcal S}|B(s,s)|.
\]

\begin{figure}[htbp]
\centering
\includegraphics[width=0.5\columnwidth]{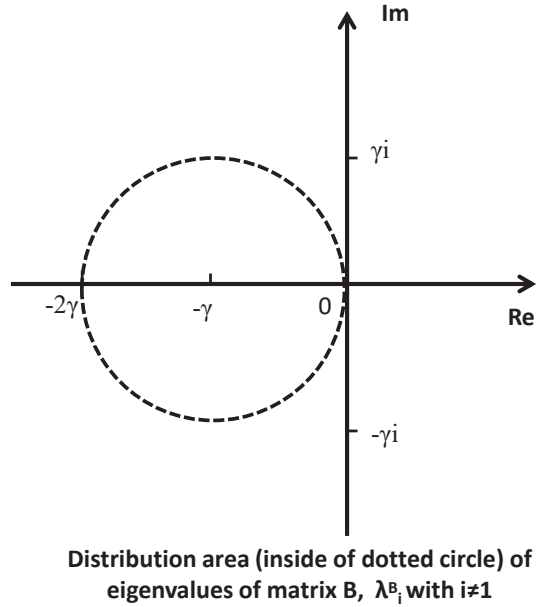}
\caption{The distribution area of eigenvalues of matrix $B$, i.e.,
$\lambda^B_i$ for $i \neq 1$.}\label{fig_circleB}
\end{figure}

Furthermore, we study the eigenvalue of matrix $\bm D := \bm B + \bm
e \bm r$ and we can similarly obtain the following theorem.
\begin{theorem}\label{theorem6}
The eigenvalue of matrix $\bm D := \bm B + \bm e \bm r$  has the
following property: $\lambda^D_1 = \bm r \bm e$ is a simple
eigenvalue and the associated column eigenvector is $\bm x^D_1 = \bm
e$; $\lambda^D_i = \lambda^B_i$ and the associated column
eigenvector is $\bm x^D_i = \bm x^B_i + \frac{\bm r \bm
x^B_i}{\lambda^B_i-\bm r \bm e}\bm e$ for $i \neq 1$ and
$\lambda^B_i \neq \bm r \bm e$.
\end{theorem}
\begin{proof}
The proof of this theorem is similar to that of
Theorem~\ref{theorem1}. We only need to verify whether the
eigenvalue and eigenvector satisfy the equation $\bm D \bm x^D_i =
\lambda^D_i \bm x^D_i$. Such verification is easy and we ignore the
details for simplicity.
\end{proof}

\begin{theorem}\label{theorem7}
For any row vector $\bm r$ satisfying $\bm r \bm e \neq 0$, the
matrix $\bm D := \bm B + \bm e \bm r$ is invertible, where $\bm B$
is the transition rate matrix of an ergodic Markov process.
\end{theorem}
\begin{proof}
Based on Theorem~\ref{theorem5} and Theorem~\ref{theorem6}, we can
easily verify that the eigenvalues of $\bm D$ satisfy:
\begin{equation}
\left\{
\begin{array}{ll}
\lambda^D_i = \bm r \bm e, &i=1;\\
|\lambda^D_i+\gamma| < \gamma, &i \neq 1;\\
\end{array}
\right.
\end{equation}
Therefore, $0$ is not the eigenvalue of the matrix $\bm D$ and $\bm
D$ is invertible.
\end{proof}

With Theorem~\ref{theorem7}, we can also simplify the computation of
the steady state distribution and the performance potential (value
function) of Markov processes.

First, we discuss how to compute the steady state distribution $\bm
\pi$ in Markov processes. It is known that $\bm \pi$ can be
determined by the following equations
\begin{equation}
\begin{array}{ll}
\bm \pi \bm B = \bm 0 \\
\bm \pi \bm e = 1\\
\end{array}
\end{equation}
Multiplying $r(i)$ on both sides of the second equation and adding
it to the $i$th row of the first equation, we can obtain the
following equation after proper manipulations
\begin{equation}
\bm \pi(\bm B + \bm e \bm r) = \bm r.
\end{equation}
With Theorem~\ref{theorem7}, for any $S$-dimensional row vector $\bm
r$ satisfying $\bm r \bm e \neq 0$, $(\bm B + \bm e \bm r)$ is
invertible and we have
\begin{center}
\begin{boxedminipage}{1\columnwidth}
\begin{equation}\label{eq_piQ}
\bm \pi = \bm r (\bm B + \bm e \bm r)^{-1}, \qquad \bm r \bm e \neq
0.
\end{equation}
\vspace{-10pt}
\end{boxedminipage}
\end{center}

Second, we discuss how to compute the performance potential of
Markov processes. In a continuous time Markov process, the
performance potential $\bm g$ is defined as below.
\begin{equation}
g(i) = \lim\limits_{T \rightarrow \infty}E\left\{ \int^{T}_{0}
[f(X_t) - \eta]dt \Big| X_0=i \right\}.
\end{equation}
We also have the Poisson equation for the above definition
\begin{equation}
-\bm B \bm g = \bm f - \eta \bm e.
\end{equation}
In the literature, it is widely adopted that $\bm g$ can be
numerically computed by the following equation \cite{Cao97,Cao07}
\begin{equation}\label{eq37}
\bm g = -(\bm B - \bm e \bm \pi)^{-1}\bm f.
\end{equation}
The above equation includes the condition $\bm \pi \bm g = \eta$.
Similar to the case of discrete time Markov chain, we can also
derive
\begin{equation}
(\bm B + \bm e \bm r)\bm g = -\bm f,
\end{equation}
where the condition $\bm r \bm g = \eta$ is required. With
Theorem~\ref{theorem7}, for any $S$-dimensional row vector $\bm r$
satisfying $\bm r \bm e \neq 0$, $(\bm B + \bm e \bm r)$ is
invertible and we have
\begin{center}
\begin{boxedminipage}{1\columnwidth}
\begin{equation}\label{eq_gQ}
\bm g = -(\bm B + \bm e \bm r)^{-1}\bm f, \qquad \bm r \bm e \neq 0.
\end{equation}
\vspace{-10pt}
\end{boxedminipage}
\end{center}
Therefore, we can directly compute the value of $\bm g$ using the
above equation, without the pre-computation of $\bm \pi$ required by
(\ref{eq37}).

On the other hand, we can also develop an online least-squares
algorithm to estimate $\bm g$ based on sample paths, which is
similar to the algorithm described in Fig.~\ref{fig_algo} for the
case of discrete time Markov chains. For simplicity, we ignore the
details.

\subsection{Poisson Equation for Q-factors}\label{subsection_Q}
It is well known that we have Poisson equation for the performance
potential or the value function in MDPs. Similar to the performance
potential quantifying the effect of the initial state on the average
performance, the Q-factor is an important quantity in reinforcement
learning and it quantifies the effect of the state-action pair on
the system average performance. Below, we give a Poisson equation
for Q-factors in an MDP under the time average criterion.

Suppose the current policy is $\mathcal L$. In this subsection, all
the quantities of MDPs are assumed for the policy $\mathcal L$ by
default, unless we have specific other notations. The Q-factor of
the Markov system under this policy $\mathcal L$ is defined as
below.
\begin{equation}\label{eq_Q1}
Q^{\mathcal L}(s,a) := f(s,a) - \eta^{\mathcal L} + \sum_{s' \in
\mathcal S} p(s,a,s')g^{\mathcal L}(s'),
\end{equation}
where $p(s,a,s')$ is the probability of which the system transits
from the current state $s$ to the next state $s'$ if the action $a$
is adopted.

For a randomized policy, $\mathcal L$ is a mapping $\mathcal L:
\mathcal S \rightarrow \Omega(\mathcal A)$, where $\Omega(\mathcal
A)$ is the set of probability measurements over the action space
$\mathcal A$. That is, $\mathcal L(s,a)$ is the probability of which
action $a$ is adopted at state $s$, $s \in \mathcal S$ and $a \in
\mathcal A$. For the deterministic case, $\mathcal L$ is a mapping
$\mathcal L: \mathcal S \rightarrow \mathcal A$, i.e., $\mathcal
L(s)$ indicates the action adopted at state $s$. We can rewrite
(\ref{eq_Q1}) in a recursive form as below.
\begin{equation}\label{eq_QPoisson}
Q^{\mathcal L}(s,a) = f(s,a) - \eta^{\mathcal L} + \sum_{s' \in
\mathcal S} p(s,a,s') \sum_{a' \in \mathcal A}\mathcal
L(s',a')Q^{\mathcal L}(s',a').
\end{equation}
The above equation is a fixed-point equation for Q-factors, which is
similar to the Poisson equation for the performance potential or the
value function in the classical MDP theory.

By sorting the element $Q^{\mathcal L}(s,a)$ in a vector form, we
define an $SA$-dimensional column vector $\bm Q^{\mathcal L}$ as
below.
\begin{equation}
\bm Q^{\mathcal L}(s) := \left[
\begin{array}{c}
Q^{\mathcal L}(s,a_1)\\
Q^{\mathcal L}(s,a_2)\\
\vdots\\
Q^{\mathcal L}(s,a_A)
\end{array}
\right], \ \forall s \in \mathcal S. \qquad \bm Q^{\mathcal L} :=
\left[
\begin{array}{c}
\bm Q^{\mathcal L}(1)\\
\bm Q^{\mathcal L}(2)\\
\vdots\\
\bm Q^{\mathcal L}(S)
\end{array}
\right].
\end{equation}

With the same order to sort the element $s,a,s',a'$, we can rewrite
(\ref{eq_QPoisson}) in a matrix form as below.
\begin{equation}\label{eq_QPoisson2}
\bm Q^{\mathcal L} = \bm f - \eta^{\mathcal L} \bm e + \bm P \bm L
\bm Q^{\mathcal L},
\end{equation}
where $\bm Q^{\mathcal L}$ and $\bm r$ are column vectors with size
$SA$, $\bm e$ is a column vector of ones with a proper size (here
the size is $SA$), $\bm P$ is a stochastic matrix with size $SA
\times S$
\begin{equation}
\bm P := \Big [\bm P((s,a),s') \Big]_{SA \times S}, \quad \bm
P((s,a),s')=p(s,a,s'), \ \forall s,s' \in \mathcal S, \ a \in
\mathcal A,
\end{equation}
$\bm L$ is a stochastic matrix with size $S \times SA$
\begin{equation}
\bm L := \Big [\bm L(s',(s',a')) \Big]_{S \times SA}, \quad \bm
L(s',(s',a'))=\mathcal L(s',a'), \ \forall s' \in \mathcal S, \ a'
\in \mathcal A,
\end{equation}
where $\bm L(s,(s',a'))=0$ if $s \neq s'$. Therefore, most of the
elements of $\bm L$ is 0 and $\bm L$ is a sparse matrix like below
\begin{equation}
\bm L = \left[
\begin{array}{ccccc}
\mathcal L(1,:) & \bm 0 & \bm 0 & \cdots & \bm 0\\
\bm 0 & \mathcal L(2,:) & \bm 0 & \cdots & \bm 0\\
\vdots & \vdots & \vdots & \ddots & \vdots\\
\bm 0 & \bm 0 & \bm 0 & \cdots & \mathcal L(S,:)
\end{array}
\right]_{S \times SA}.
\end{equation}
If we write $\bm {\mathcal L}$ as an $S \times A$ matrix as below.
\begin{equation}
\bm {\mathcal L} := \left[
\begin{array}{cccc}
\mathcal L(1,a_1),&\mathcal L(1,a_2),&\cdots,&\mathcal L(1,a_A)\\
\mathcal L(2,a_1),&\mathcal L(2,a_2),&\cdots,&\mathcal L(2,a_A)\\
\vdots & \vdots & \ddots & \vdots\\
\mathcal L(S,a_1),&\mathcal L(S,a_2),&\cdots,&\mathcal L(S,a_A)\\
\end{array}
\right]_{S \times A}.
\end{equation}
Then matrix $\bm L$ equals the block diagonal of \emph{Kronecker
product} (matrix form of tensor-product) of vector $\bm e^T$ and
matrix $\bm {\mathcal L}$, where $\bm e^T$ is an $S$-dimensional row
vector of ones. That is, we have
\begin{equation}
\bm L = \diag (\bm e^T \otimes \bm {\mathcal L})  = \left[
\begin{array}{cccc}
\bm {\mathcal L}(1,:),&\bm 0,&\cdots,&\bm 0\\
\bm 0,&\bm {\mathcal L}(2,:),&\cdots,&\bm 0\\
\vdots & \vdots & \ddots & \vdots\\
\bm 0,&\bm 0,&\cdots,&\bm {\mathcal L}(S,:)\\
\end{array}
\right]_{S \times SA}.
\end{equation}
It is easy to verify that
\begin{equation}
\begin{array}{l}
\bm L \bm e = \bm e. \\
\bm P \bm e = \bm e.
\end{array}
\end{equation}
Therefore,
\begin{equation}
\bm P \bm L \bm e = \bm e,
\end{equation}
and $\bm P \bm L$ is still a stochastic matrix that can be denoted
as matrix $\bm {\tilde{P}}$ with size $SA \times SA$. That is,
\begin{equation}
\bm {\tilde{P}} := \bm P \bm L.
\end{equation}
Therefore, we obtain a linear form of (\ref{eq_QPoisson2}) as below.
\begin{center}
\begin{boxedminipage}{1\columnwidth}
\begin{equation}\label{eq_QPoisson3}
(\bm I - \bm {\tilde{P}} )\bm Q^{\mathcal L} = \bm f -
\eta^{\mathcal L} \bm e.
\end{equation}
\vspace{-10pt}
\end{boxedminipage}
\end{center}

We can see that for any solution of $\bm Q^{\mathcal L}$ satisfying
(\ref{eq_QPoisson3}), $\bm Q^{\mathcal L} + c \bm e$ is still a
solution to (\ref{eq_QPoisson3}), where $c$ is any constant.
Therefore, we can let $\bm Q^{\mathcal L}$ satisfy
\begin{equation}
\bm r \bm Q^{\mathcal L} = \eta,
\end{equation}
where $\bm r$ is any $SA$-dimensional row vector satisfying $\bm r
\bm e \neq 0$. Substituting the above equation into
(\ref{eq_QPoisson3}), we obtain
\begin{equation}\label{eq_QPoisson4}
(\bm I - \bm {\tilde{P}}  + \bm e \bm r)\bm Q^{\mathcal L} = \bm f.
\end{equation}

Since $\bm {\tilde{P}} $ is a stochastic matrix, with
Theorem~\ref{theorem2}, we know that $(\bm I - \bm {\tilde{P}} + \bm
e \bm r)$ is invertible. Therefore, we have the following solution
of Q-factors
\begin{center}
\begin{boxedminipage}{1\columnwidth}
\begin{equation}\label{eq_QPoisson5}
\bm Q^{\mathcal L} = (\bm I - \bm {\tilde{P}}  + \bm e \bm r)^{-1}
\bm f, \qquad \bm r \bm e \neq 0.
\end{equation}
\vspace{-10pt}
\end{boxedminipage}
\end{center}

Therefore, we obtain the Poisson equation (\ref{eq_QPoisson3}) for
Q-factors, which is a fixed point equation to solve Q-factors. The
closed-form solution of Q-factors is also obtained in
(\ref{eq_QPoisson5}). Based on these equations, we may also develop
numerical computation algorithms or online estimation algorithms for
Q-factors, similar to the discussion and algorithms in
Subsection~\ref{subsection_Alg}. Recently, the deep reinforcement
learning, such as the AlphaGo of Google, is becoming a promising
direction of artificial intelligence \cite{Silver16} and the
Q-factors are fundamental quantities in reinforcement learning
\cite{Sutton98}, how to efficiently compute, estimate or even
represent the Q-factors is an interesting topic that deserves
further research efforts.

\section{Conclusion} \label{section_Conclusion}
In this paper, we study a generalized fundamental matrix in Markov
systems. Different from the fundamental matrix in the classical MDP
theory, the generalized fundamental matrix does not require the
pre-computation of $\bm \pi$ and it can provide a more concise form
for the computation of some fundamental quantities in Markov
systems. Based on the generalized fundamental matrix, we give a
closed-form solution to the Poisson equation and represent the
values of performance potentials, steady state distribution, and
Q-factors of Markov systems. The new representation of these
solutions may shed some light on efficiently computing or estimating
these fundamental quantities from a new perspective, which is very
important for the performance optimization of Markov decision
processes.

\section*{Acknowledgment}
The first author was supported in part by the National Natural
Science Foundation of China (61573206, 61203039, U1301254) and would
like to thank X. R. Cao, J. J. Hunter, C. D. Meyer, and M. L.
Puterman for their helpful discussions and comments.

\end{document}